\documentclass[11pt,leqno]{amsart}
\usepackage{amsmath,amssymb,amsthm}
\usepackage{hyperref}

\usepackage{bbm}

\newcommand{\R}{\mathbb{R}}

\DeclareMathOperator{\diam}{diam\,}
\DeclareMathOperator{\co}{co}

\renewcommand{\geq}{\geqslant}
\renewcommand{\leq}{\leqslant}

\newcommand{\Lip}{{\mathrm{Lip}}_0}
\newcommand{\justLip}{{\mathrm{Lip}}}

\newcommand{\ext}[1]{\operatorname{ext}\left(#1\right)}
\newcommand{\strexp}[1]{\operatorname{str-exp}\left(#1\right)}
\newcommand{\preext}[1]{\operatorname{pre-ext}\left(#1\right)}
\newcommand{\dent}[1]{\operatorname{dent}\left(#1\right)}

\newtheorem{theorem}{Theorem}[section]
\newtheorem{lemma}[theorem]{Lemma}
\newtheorem{claim}[theorem]{Claim}
\newtheorem{proposition}[theorem]{Proposition}

\theoremstyle{definition}

\newtheorem{example}[theorem]{Example}

\theoremstyle{remark}
\newtheorem{remark}[theorem]{Remark}
\numberwithin{equation}{section}

\def\fnote#1{\footnote}

\def\ignora#1{}
\def\n3#1{\left\vert  \! \left\vert \! \left\vert \, #1 \, \right\vert \!
  \right\vert \! \right\vert }

\begin{document}

\title{ Some results on isometric composition operators on Lipschitz spaces }

\author{ Abraham Rueda Zoca }\thanks{The research of Abraham Rueda Zoca was supported by Vicerrectorado de Investigaci\'on y Transferencia de la Universidad de Granada in the program ``Contratos puente'', by MICINN (Spain) Grant PGC2018-093794-B-I00 (MCIU, AEI, FEDER, UE), by Junta de Andaluc\'ia Grant A-FQM-484-UGR18
and by Junta de Andaluc\'ia Grant FQM-0185}
\address{Universidad de Granada, Facultad de Ciencias.
Departamento de An\'{a}lisis Matem\'{a}tico, 18071-Granada
(Spain)} \email{ abrahamrueda@ugr.es}
\urladdr{\url{https://arzenglish.wordpress.com}}



\maketitle \markboth{Abraham Rueda Zoca}{
  Some results on isometric composition operators on Lipschitz spaces }

\begin{abstract}
Given two metric spaces $M$ and $N$ we study, motivated by a question of N. Weaver, conditions under which an isometric composition operator $C_\phi:\Lip(M)\longrightarrow \Lip(N)$ is isometric depending on the properties of $\phi$. We obtain a complete characterisation of those operators $C_\phi$ in terms of a property of the function $\phi$ in the case that $B_{\mathcal F(M)}$ is the closed convex hull of its preserved extreme points. Also, we obtain necessary and sufficient conditions for $C_\phi$ being isometric in the case that $M$ is geodesic.
\end{abstract}

\section{Introduction}

In this paper we will analyse the question of when a composition operator between spaces of Lipschitz functions is isometric. Let us start with necessary definitions. A \emph{pointed metric space} is just a metric space $M$ in which we distinguish an element, called $0$. Given a pointed metric space $M$, we write $\Lip(M)$ to denote the Banach space of all Lipschitz maps $f:M\longrightarrow \mathbb R$ which vanish at $0$, endowed with the Lipschitz norm defined by
$$ \| f \| := \sup\left\{\frac{f(x)-f(y)}{d(x,y)} \colon x,y\in M,\, x \neq y \right\}.$$
Given two pointed metric spaces $M$ and $N$ and a Lipschitz map $\phi:N\longrightarrow M$ such that $\phi(0)=0$, then $\phi$ induces a \textit{composition operator} $C_\phi:\Lip(M)\longrightarrow \Lip(N)$ given by the equation
$$C_\phi(f)=f\circ \phi.$$
The study of this kind of operator is very present in the first version of \cite{wea} in an effort to give a characterisation of those onto linear isometries between spaces of Lispchitz functions. The study of the \textit{Lipschitz-free spaces} and its Banach space structure resulted into a very useful tool to treat this problem. Let us formally introduce these spaces.

Let $M$ be a pointed metric space. We denote by $\delta$ the canonical isometric embedding of $M$ into $\Lip(M,\R)^*$, which is given by $\langle f, \delta(x) \rangle =f(x)$ for $x \in M$ and $f \in \Lip(M,\R)$. We denote by $\mathcal{F}(M)$ the norm-closed linear span of $\delta(M)$ in the dual space $\Lip(M,\R)^*$, which is usually called the \textit{Lipschitz-free space over $M$}, see the papers \cite{god15} and \cite{gk}, and the book \cite{wea} (where it receives the name of Arens-Eells space) for background on this. It is well known that $\mathcal{F}(M)$ is an isometric predual of the space $\Lip(M,\R)$ \cite[pp. 91]{god15}. We will denote by $\delta_x:=\delta(x)$ and we will consider a \textit{molecule in $\mathcal F(M)$} as an element of the form
$$m_{x,y}:=\frac{\delta_x-\delta_y}{d(x,y)}$$
for $x,y\in M$ such that $x\neq y$.

Coming back to Lipschitz functions, let us recall that when $M$ and $N$ are pointed metric space, it is well known that every Lipschitz function $f \colon N \longrightarrow M$ which preserves the origin can be isometrically identified with the continuous linear operator $\widehat{f} \colon \mathcal{F}(N) \longrightarrow \mathcal F(M)$ defined by $\widehat{f}(\delta_p)=\delta_{f(p)}$ for every $p \in M$. This mapping completely identifies the spaces $\Lip(M,Y)$ and $\mathcal{L}(\mathcal{F}(M),Y)$. 

Anyway, under this point of view, if $\phi:N\longrightarrow M$ is a Lipschitz mapping, then $C_\phi$ is nothing but the adjoint operator of $\hat\phi$. That is the reason why the space $\mathcal F(M)$ in general and the extremal structure of its unit ball in particular are extremely useful when dealing with operators between spaces of Lipschitz functions. For instance, by making use of the study of the preserved extreme points of $B_{\mathcal F(M)}$, it is characterised in \cite[Theorem 3.56]{wea} the surjective linear isometries between spaces of Lipschitz functions over uniformly convave metric spaces (see \cite[Theorem 3.56]{wea} for details).

In connection with the composition operators, N. Weaver wondered in \cite[pp. 53]{wea} by a characterisation of those $C_\phi$ which are isometries in terms of a condition on the defining Lipschitz function $\phi$. Very recently, A. Jim\'enez-Vargas obtained in \cite{jiva} a characterisation in the following sense: if $\phi:N\longrightarrow M$ is a norm-one Lipschitz function and $M$ has the so-called \textit{peaking property}, then $C_\phi:\Lip(M)\longrightarrow \Lip(N)$ is isometric if, and only if, for every pair of points $x,y\in M, x\neq y$ we can find sequences $x_n\subseteq N$ and $y_n\subseteq N$ such that $\phi(x_n)\rightarrow x$, $\phi(y_n)\rightarrow y$ and 
$$\frac{d(\phi(x_n),\phi(y_n))}{d(x_n,y_n)}\rightarrow 1.$$
To see how extremal structure of $\mathcal F(M)$ appears in the above mentioned theorem, let us explain what is the peaking property. We say that $M$ has the peaking property if, for every pair of distinct points $x,y\in M, x\neq y$, there exists a function $f_{x,y}\in S_{\Lip(M)}$ (which is said to \textit{peak the pair $(x,y)$}) satisfying that $f_{x,y}(m_{x,y})=1$ and so that, for every open set $U$ in $M^2\setminus \Delta=\{(x,y)\in M^2: x\neq y\}$ containing $(x,y)$ and $(y,x)$, then there exists a $\delta>0$ such that
$$\vert f_{x,y}(m_{u,v)}\vert\leq 1-\delta$$
if $(u,v)\notin U$.

The connection between this property and the extremal structure of $\mathcal F(M)$ is \cite[Theorem 5.4]{gpr}, where it is proved that a function $f$ peakes a pair $(x,y)$ if, and only if, the element $m_{x,y}$ is a strongly exposed point in $\mathcal F(M)$ and $\hat f_{x,y}$ is a strongly exposing functional for $m_{x,y}$. So, in the language of \cite[Theorem 5.4]{gpr}, the peaking property of $M$ can be reformulated as the fact that $m_{x,y}$ is a strongly exposed point of $B_{\mathcal F(M)}$ for every pair of different points $x,y\in M$.

With the previous information in mind, the main aim of Section \ref{section:dientes} is to generalise the above mentioned result \cite[Theorem 2.4]{jiva} and to prove that, if $\mathcal F(M)$ is the closed convex hull of its preserved extreme points (see formal definition below), then given a norm-one Lipschitz function $\phi:N\longrightarrow M$ we have that $C_\phi$ is isometric if, and only if, for every pair of different points $x,y\in M$ such that $m_{x,y}$ is a preserved extreme point, we can find a pair of sequences $x_n$ and $y_n$ in $N$ so that $\phi(x_n)\rightarrow x, \phi(y_n)\rightarrow y$ and
$$\frac{d(\phi(x_n),\phi(y_n))}{d(x_n,y_n)}\rightarrow 1.$$
This proves that the equivalence established in \cite[Theorem 2.4]{jiva} actually works for a rather larger class of metric spaces $M$ (see Example \ref{exam:aplidientes}).

To shorten, we can say that in Section \ref{section:dientes} we are studying composition operators $C_\phi:\Lip(M)\longrightarrow \Lip(N)$ for metric spaces $M$ such that $\mathcal F(M)$ has a rich extremal structure (the set of denting points of the unit ball of $B_{\mathcal F(M)}$ is norming). In Section \ref{section:length} we aim to study the extremely oposite case, that is, the case when $\mathcal F(M)$ does not contain any denting point. According to \cite[Theorem 1.5]{am}, given a complete metric space $M$, then the unit ball of $\mathcal F(M)$ does not have any preserved extreme point if, and only if, $M$ is \textit{length} (see formal definition below). Because of this reason, we will first study the composition operators $C_\phi:\Lip([0,1])\longrightarrow \Lip(N)$. Though we do not obtain a complete characterisation in this case, we will obtain necessary and sufficient conditions in Propositions \ref{proponece[0,1]} and \ref{prop:sufi[0,1]} which are closely related to an abundance of points in $N$ where, roughly speaking, $\phi$ has derivative exactly one. To be more precise, we get for instance in Proposition \ref{prop:sufi[0,1]}, that $C_\phi$ is isometric if $\phi(N)$ has length one and, for every $t\in\phi(N)$, there exists $x\in N$ such that $\phi(x)=t$ and that
$$\limsup\limits_{y\rightarrow x, y\neq x}\frac{d(\phi(y),\phi(x))}{d(y,x)}=1.$$
Finally, in Theorems \ref{theo:maincondinece} and \ref{theo:maincondisufi} we obtain necessary and sufficient conditions, which are closely related to the case of $M=[0,1]$, under which $C_\phi:\Lip(M)\longrightarrow \Lip(M)$ is an isometry when $M$ is geodesic covering, in particular, all the length metric spaces which are compact.

\bigskip

\textbf{Notation:} We will only consider real Banach spaces. Given a Banach space $X$ we will denote by $B_X$ and $S_X$ the closed unit ball and the closed unit sphere. Also, $X^*$ stands for the topological dual of $X$. A \textit{slice} of the unit ball $B_X$ is a non-empty intersection of an open half-space with $B_X$; every slice can be written in the form \[
S(B_X,f,\beta):=\{x\in B_X \colon f(x)>1-\beta\},
\]
where $f \in S_{X^*}$, $\beta>0$.

The notations $\ext{B_X}$, $\preext{B_X}$, $\strexp{B_X}$ stand for the set of extreme points, preserved extreme points (i.e.\ extreme points which remain extreme in the bidual ball), and strongly exposed points of $B_X$, respectively. A point $x \in B_X$ is said to be a \emph{denting point} of $B_X$ if there exist slices of $B_X$ containing $x$ of arbitrarily small diameter. We will denote by $\dent{B_X}$ the set of denting points of $B_X$. We always have that
$$
\strexp{B_X}\subset \dent{B_X} \subset \preext{B_X} \subset \ext{B_X}.
$$ 
The study of the extremal structure in the paricular case of being $X$ a Lipschitz-free space has experimented a recent and intense research (see e.g. \cite{ag,ap,gppr,gpr}). Among all this research, particularly useful in the main result of Section \ref{section:dientes} is \cite[Theorem 2.4]{gppr}, which establishes that every preserved extreme point is a denting point in a Lipschitz-free space.

Given a metric space $M$, we say that $M$ is \textit{length} if, for every pair of distinct points $x,y\in M$, then $d(x,y)$ is equal to the infimum of the length of the rectifiable curves joining them. If such infimum is actually a minimum for every pair of points we will say that $M$ is \textit{geodesic}. See \cite{bh} for background on length spaces. In the context of Lipschitz-free spaces, these notions have been used in connection with the \textit{Daugavet property} in the papers \cite{am,gpr,ikw}.

Let $M$ be a metric space and $f:M\longrightarrow \mathbb R$ be a Lipschitz function. According to \cite{duja}, the \textit{pointwise Lipschitz constant of $f$} at a non-isolated point $x\in M$ is defined as
$$\justLip f(x):=\limsup\limits_{y\rightarrow x, y\neq x}\frac{f(y)-f(x)}{d(y,x)},$$
and it is defined $\justLip f(x)=0$ if $x$ is an isolated point. See \cite{duja} and references therein for background on pointwise Lipschitz constants.

Let us end the section with some notation about \textit{generalised derivatives} which will be used in Example \ref{exam:clarke}. Let $X$ be a Banach space and $f:X\longrightarrow \mathbb R$ a
Lipschitz function. According to \cite{cla}, the \textit{generalized derivative of at a point $x\in X$ in the direction $v\in X$} is defined by
$$f^\circ(x,v):=\limsup\limits_{ y\rightarrow x,
t\searrow 0}\frac{f(y+tv)-f(y)}{t}.$$
Such a limit always exists from the Lipschitz condition. Moreover, it is a sublinear and positively homogeneous function in the variable $v$ \cite[Proposition 2.1.1]{cla}. In addition, the generalized gradient of $f$ at $x$ is defined as
follows
$$\partial f(x):=\{x^*\in X^*: f^\circ(x,v)\geq x^*(v)\ \forall v\in X\}.$$
Given $v\in X$ it follows that \cite[Proposition 2.1.2]{cla}
$$f^\circ(x,v)=\max\limits_{x^*\in \partial f(x)}x^*(v)\ \ \forall x\in X.$$
According to \cite[Definition 2.3.4]{cla}, $f$ is said to be \textit{regular at $x$} if:
\begin{enumerate}
\item For every $v\in X$, the directional derivative $f'(x,v)$ exists, and,
\item $f'(x,v)=f^\circ (x,v)$ holds for every $v\in X$.
\end{enumerate}
We refer to \cite[Proposition 2.6.6]{cla} for examples.

\section{Isometric composition operators and denting points}\label{section:dientes}

Let $M$ and $N$ be two (complete) pointed metric spaces and $\phi:N\longrightarrow M$ be a norm-one Lipschitz function such that $\phi(0)=0$. Let $C_\phi: \Lip(M)\longrightarrow \Lip(N)$ given by $C_\phi(f)=f\circ \phi$. We wonder under which conditions $C_\phi$ is an isometry. To begin with, let us start with the following result, which is a slight modification of \cite[Theorem 2.1]{jiva} which will be used in the sequel.

\begin{proposition}\label{propocaranorma}
Let $A\subseteq M^2\setminus\Delta$ such that:
\begin{enumerate}
\item The set $\{m_{x,y}: (x,y)\in A\}$ is norming for $\Lip(M)$.
\item For every $(x,y)\in A$ there exists a pair of sequences $\{x_n\},\{y_n\}$ in $N$ such that $\phi(x_n)\neq \phi(y_n)$ holds for every $n\in\mathbb N$, $\phi(x_n)\rightarrow x, \phi(y_n)\rightarrow y$ and 
$$\frac{d(\phi(x_n),\phi(y_n))}{d(x_n,y_n)}\rightarrow 1.$$
\end{enumerate}
Then $C_\phi$ is an isometry.
\end{proposition}

\begin{proof}
Let $f\in S_{\Lip(M)}$. Clearly 
$$\Vert C_\phi(f)\Vert=\Vert f\circ\phi\Vert\leq \Vert f\Vert \Vert \phi\Vert=\Vert f\Vert.$$
In order to prove the reverse inequality, pick $\varepsilon>0$ and choose $(x,y)\in A$ such that $f(m_{x,y})>1-\varepsilon=\Vert f\Vert-\varepsilon$. By assumptions we can find a pair of sequences $\{x_n\}, \{y_n\}$ in $N$ such that $\phi(x_n)\rightarrow x, \phi(y_n)\rightarrow y$ and 
$$\frac{d(\phi(x_n),\phi(y_n))}{d(x_n,y_n)}\rightarrow 1.$$
Now
\[
\begin{split}
\Vert f\circ \phi\Vert& \geq \frac{f(\phi(x_n))-f(\phi(y_n))}{d(x_n,y_n)}\\
& = \frac{f(\phi(x_n))-f(\phi(y_n))}{d(\phi(x_n),\phi(y_n))} \frac{d(\phi(x_n),\phi(y_n)}{d(x_n,y_n)}\\
& \rightarrow \frac{f(x)-f(y)}{d(x,y)}>\Vert f\Vert-\varepsilon.
\end{split}
\]
Since $\varepsilon>0$ was arbitrary we conclude that $\Vert f\circ\phi\Vert\geq \Vert f\Vert$, and we are done.
\end{proof}

In \cite[Theorem 2.4]{jiva} it is characterised the isometric composition operators $C_\phi:\Lip(M)\longrightarrow \Lip(N)$ when $M$ satisfies the peak property which, according to \cite[Theorem 5.4]{gpr}, means nothing but every molecule $m_{x,y}\in \mathcal F(M)$ is a strongly exposed point. A generalisation of the above result is the following theorem.

\begin{theorem}
Assume that $B_{\mathcal F(M)}=\overline{\co}(\preext{B_{\mathcal F(M)}})$. Then $C_\phi$ is an isometry if, and only if, for every $(x,y)\in M^2\setminus\Delta$ such that $m_{x,y}\in \preext{B_{\mathcal F(M)}}$, there exists a pair of sequences $\{x_n\}, \{y_n\}\in N$ such that $\phi(x_n)\rightarrow x, \phi(y_n)\rightarrow y$ and 
$$\frac{d(\phi(x_n),\phi(y_n))}{d(x_n,y_n)}\rightarrow 1.$$
\end{theorem}

\begin{proof} The ``if'' part follows from Proposition \ref{propocaranorma}. For the converse, given $(x,y)\in M^2\setminus \Delta$ such that $m_{x,y}$ is a preserved extreme point, then $m_{x,y}$ is a denting point \cite[Theorem 2.4]{gppr}. Hence, for every $n\in\mathbb N$, we can find $f_n\in S_{\Lip(M)}$ and $\beta_n>0$ such that $f_n(m_{x,y})>1-\beta_n$ and that
$$\diam(S(B_{\mathcal F(M)},f_n,\beta_n))<\frac{1}{n}.$$
Now, given $n\in\mathbb N$, it follows that $\Vert C_\phi(f_n)\Vert=\Vert f_n\Vert=1$. Consequently, we can find a pair of sequences $x_k^n, y_k^n$ in $N$ such that 
$$\frac{f_n(\phi(x_k^n))-f_n(\phi(y_k^n))}{d(x_k^n,y_k^n)}\mathop{\longrightarrow}\limits^{k\rightarrow \infty} 1.$$
Since, for every $k\in\mathbb N$, we have 
$$\frac{f_n(\phi(x_k^n))-f_n(\phi(y_k^n))}{d(x_k^n,y_k^n)}\leq \frac{f(\phi(x_k^n))-f(\phi(y_k^n))}{d(\phi(x_k^n),\phi(y_k^n))} \frac{d(\phi(x_k^n),\phi(y_k^n))}{d(x_k^n,y_k^n)},$$
we get that both of the previous factors converge to 1. Hence, for every $n\in\mathbb N$, we can find $\sigma(n)\in\mathbb N$ such that
$$\frac{d(\phi(x_{\sigma(n)}^n),\phi(y_{\sigma(n)}^n))}{d(x_{\sigma(n)}^n,y_{\sigma(n)}^n)}>1-\frac{1}{n},$$
and
$$\frac{f_n(\phi(x_{\sigma(n)}^n)-f_n(\phi(y_{\sigma(n)}^n))}{d(\phi(x_{\sigma(n)}^n),\phi(y_{\sigma(n)}^n))}>1-\beta_n.$$
Now the second condition implies that $m_{\phi(x_{\sigma(n)}^n),\phi( y_{\sigma(n)}^n)}\in S(B_{\mathcal F(M)},f_n,\beta_n)$, which has diameter smaller than $\frac{1}{n}$. Consequently $\Vert m_{\phi(x_{\sigma(n)}^n),\phi( y_{\sigma(n)}^n)}-m_{x,y}\Vert<\frac{1}{n}$. Now \cite[Lemma 4.1.13]{lctesis} implies that 
$$\frac{1}{n}>\Vert m_{\phi(x_{\sigma(n)}^n),\phi( y_{\sigma(n)}^n)}-m_{x,y}\Vert\geq \frac{\max\{d(\phi(x_{\sigma(n)}^n),x),
d(\phi(y_{\sigma(n)}^n),y) \}}{d(x,y)}$$
holds for every $n\in\mathbb N$. So, taking $x_n:= x_{\sigma(n)}^n$ and $y_n:=y_{\sigma(n)}^n$ we get that $\phi(x_n)\rightarrow x, \phi(y_n)\rightarrow y$ and $\frac{d(\phi(x_n),\phi(y_n))}{d(x_n,y_n)}\rightarrow 1$, and we are done.
\end{proof}

Let us see examples below where the previous theorem applies.

\begin{example}\label{exam:aplidientes}
$B_{\mathcal F(M)}=\overline{\co}(\preext{B_{\mathcal F(M)}})$ in the following cases:
\begin{enumerate}
    \item If $\mathcal F(M)$ has the RNP. In particular, when $M$ is compact and H\"older \cite{wea} or when $M$ is uniformy discrete \cite{kalton}.
    \item If $M$ is the unit circle in $\mathbb R^2$ \cite[Theorem 2.1]{cgmr}.
    \item If $M$ is boundedly compact (i.e. if every closed ball in $M$ is compact) and $\operatorname{SNA}(M,\mathbb R)$, the set of those Lipchitz functions which strongly attain its norm (see \cite{god15,kms} for background), is dense in $\Lip(M,\mathbb R)$ \cite[Corollary 3.21]{cgmr}.
\end{enumerate}
\end{example}

\section{Isometric composition operators into geodesic spaces}\label{section:length}

Let us start with a study of the isometric composition operators $\phi:N\longrightarrow [0,1]$. A neccesary condition is established in the following proposition.

\begin{proposition}\label{proponece[0,1]}
Let $N$ be a metric space and $\phi:N\longrightarrow [0,1]$ be a 1-Lipschitz function. If $C_\phi$ is an isometry, then for every $x\in [0,1]$ we can find a pair of sequences $x_n, y_n$ in $N$ such that
\begin{enumerate}
    \item $\phi(x_n)\rightarrow x$ and $\phi(y_n)\rightarrow x$.
    \item $\frac{d(\phi(x_n),\phi(y_n))}{d(x_n,y_n)}\rightarrow 1.$
\end{enumerate} \end{proposition}

\begin{proof}
Pick $x\in [0,1]$. Consider $g:[0,1]\longrightarrow [0,1]$ defined by the equation 
$$g(t):=\int_x^t  1-\vert x-s\vert\ ds,$$ which is a norm-one Lipschitz function. Also, notice that if $u_n, v_n\in [0,1]$ satisfies that $g(m_{u_n,v_n})\rightarrow 1$ then $u_n\rightarrow x$ and $v_n\rightarrow x$. Since we are assuming that $C_\phi$ is an isometry we can find $x_n, y_n\in N$ such that $C_\phi(g)(m_{x_n,y_n})\rightarrow 1$. Now
$$\frac{g(\phi(x_n))-g(\phi(y_n))}{d(x_n,y_n)}\leq \frac{g(\phi(x_n))-g(\phi(y_n))}{\vert \phi(x_n)-\phi(y_n)\vert}\frac{\vert \phi(x_n)-\phi(y_n)\vert}{d(x_n,y_n)}.$$
This implies that
\begin{equation}\label{prop:nece01condi1}
   \frac{g(\phi(x_n))-g(\phi(y_n))}{\vert \phi(x_n)-\phi(y_n)\vert}\rightarrow 1 
\end{equation}
and
\begin{equation}\label{prop:nece01condi2}
   \frac{\vert \phi(x_n)-\phi(y_n)\vert}{d(x_n,y_n)}\rightarrow 1. 
\end{equation}
Now the second condition of the thesis of the proposition is simply \eqref{prop:nece01condi2}, whereas the first one follows from \eqref{prop:nece01condi1} and the property exhibited of the function $g$.
\end{proof}

If we consider a slight strenghthening on the second condition in Proposition \ref{proponece[0,1]} we arrive at a sufficient condition for a composition operator to be isometric.

\begin{proposition}\label{prop:sufi[0,1]}
Let $N$ be a metric space and let $\phi:N\longrightarrow [0,1]$ be a 1-Lipschitz function so that $\phi(N)$ has length $1$. Assume that for every $t\in\phi(N)$ there exists $x\in N$ such that $\phi(x)=t$ and that
$$\justLip \phi(x)=1.$$
Then $C_\phi$ is an isometry.
\end{proposition}

\begin{proof}
Pick a norm-one Lipschitz function $f\in \Lip(M)$. Let us see that $\Vert f\circ\phi\Vert=1$. To this end pick $\varepsilon>0$ and choose $x\neq y\in M$ such that 
$$\frac{f(x)-f(y)}{d(x,y)}=\frac{1}{d(x,y)}\int_y^x f'(t)\ dt >1-\varepsilon^2.$$
By this inequality and since $\lambda(\phi(N))=1$ we can find $t\in \phi(N)$ such that $f'(t)>1-\varepsilon$. By assumptions we can find $x\in N$ such that $\phi(x)=t$ and a sequence $\{x_n\}\subseteq N$ such that $x_n\neq x$ holds for all $n\in\mathbb N$, $\phi(x_n)\rightarrow \phi(x)$ and 
$$\frac{\phi(x_n)-\phi(x)}{d(x_n,x)}\rightarrow 1.$$ 
Since $\phi(x_n)\rightarrow x$ and they are different we get that 
$$\frac{f(\phi(x_n))-f(\phi(x))}{d(\phi(x_n),\phi(x))}\rightarrow f'(t)>1-\varepsilon.$$
Hence
$$\Vert C_\phi(f)\Vert\geq \frac{C_\phi(f)(x_n)-C_\phi(f)(x)}{d(x_n,x)}$$

$$\geq \frac{f(\phi(x_n))-f(\phi(x))}{\phi(x_n)-\phi(x)}\frac{\phi(x_n)-\phi(x)}{d(x_n,x)}\rightarrow f'(t)>1-\varepsilon.$$
Since $\varepsilon>0$ was arbitrary we conclude the desired result.
\end{proof}

We do not know whether the converse of Proposition \ref{proponece[0,1]} holds. Let us exhibit, however, a class of metric spaces $N$ and of Lipschitz functions $\phi$ where Proposition \ref{proponece[0,1]} reverses.

\begin{example}\label{exam:clarke}
Let $X$ be a finite-dimensional Banach space, let $C$ be a bounded, open and convex subset of $X$ such that every point of $\partial C$ has a unique supporting tangent hyperspace. Let $N:=\overline{C}$. Consider $f:X\longrightarrow \mathbb R$ be a norm-one Lipschitz function which regularises every point of $X$ (in particular, this happen if $f$ is convex) such that $f(N)=[0,1]$. Let $\phi:=f_{|N}:N\longrightarrow [0,1]$. Then $C_\phi$ is isometric if, and only if, for every $t\in [0,1]$ we can find a pair of sequences $x_n, y_n$ in $N$ such that
\begin{enumerate}
    \item $\phi(x_n)\rightarrow t$ and $\phi(y_n)\rightarrow t$.
    \item $\frac{d(\phi(x_n),\phi(y_n))}{d(x_n,y_n)}\rightarrow 1.$
    \end{enumerate}
\end{example}

\begin{proof}
The neccesity is just Proposition \ref{proponece[0,1]}. To prove the sufficiency, let us prove that $\phi$ satisfies the assumptions of Proposition \ref{prop:sufi[0,1]}. Pick any $t\in [0,1]$. Then, by assumptions, we can find a pair of sequences $x_n, y_n$ in $N$ such that
\begin{enumerate}
    \item $\phi(x_n)\rightarrow t$ and $\phi(y_n)\rightarrow t$.
    \item $\frac{d(\phi(x_n),\phi(y_n))}{d(x_n,y_n)}\rightarrow 1.$
    \end{enumerate}
By \cite[Theorem 2.3.7]{cla} we have that, for every $n\in\mathbb N$, we can find $u_n\in ]x_n,y_n[\subseteq N$ and $\varphi_n\in \partial f(u_n)$ such that
$$f(x_n)-f(y_n)=\varphi_n(x_n-y_n),$$
and so $\frac{f(x_n)-f(y_n)}{d(x_n,y_n)}=\varphi_n\left(\frac{x_n-y_n}{d(x_n,y_n)} \right)$ holds for every $n\in\mathbb N$. Since $N$ is compact, the sequence $\varphi_n$ is bounded in $X^*$ \cite[Proposition 2.1.2]{cla} and $\frac{x_n-y_n}{d(x_n,y_n)}$ is a sequence in the compact set $S_X$ we can assume, up taking a suitable subsequence, that $x_n\rightarrow x\in N$, $y_n\rightarrow y\in N$, $\varphi_n\rightarrow \varphi\in X^*$ and $\frac{x_n-y_n}{d(x_n,y_n)}\rightarrow v\in S_X$. Now, the assumptions on the sequences $x_n$ and $y_n$ imply that $x=y$ and that $\phi(x)=t$. Also, since $u_n\in ]x_n,y_n[$, we get that $u_n\rightarrow x$. Hence, \cite[Proposition 2.1.5]{cla} implies that $\varphi\in \partial f(x)$. Finally, notice that $\varphi(v)=1$. Since $\varphi\in \partial f(x)$, Proposition 2.1.2 in \cite{cla} implies that
$$1=\varphi(v)\leq f^\circ (x,v)\leq 1.$$
Now let us prove the following claim.

\begin{claim}
For every $\varepsilon>0$ there exists $v_\varepsilon\in S_X$ such that $f^\circ (x,v_\varepsilon)>1-\varepsilon$ and such that there exists a sequence of positive numbers $t_n\rightarrow 0$ such that $x+t_n v_\varepsilon\in N$ holds for every $n\in\mathbb N$.
 \end{claim}

\begin{proof}[Proof of the Claim.] If there exists a sequence of positive numbers $t_n\rightarrow 0$ such that $x+t_n v\in N$ holds for every $n\in\mathbb N$ then we are done. Otherwise $\{x+tv: t\in\mathbb R\}$ is contained in the (unique) supporting hyperplane, say $H$ (in particular, $x\in \partial C$). Pick $z\in C$ and notice that, for every $\varepsilon>0$ the set $x+\mathbb R (v+\varepsilon z)$ is not in $H$, so we can find a sequence of positive numbers $t_n\rightarrow 0$ such that $x+t_n(v+\varepsilon z)\in N$. Since the function $f^\circ (x,\cdot)$ is Lipschitz \cite[Proposition 2.1.2]{cla}, we get
$$f^\circ (x,v+\varepsilon z)=f^\circ (x,v)+f^\circ (x,v+\varepsilon z)-f^\circ (f,v)\geq 1-\varepsilon \Vert z\Vert.$$
Since $z\in C$ and $C$ is bounded, the Claim is proved. \end{proof}
Since $f$ regularises $x$ in the language of \cite{cla}, which means that the classical directional derivative $f'(x,v_\varepsilon)$ exists and agrees with $f^\circ (x,v)$. Consequently, 
$$1-\varepsilon<\lim\limits_{n}\frac{f(x+t_n v_\varepsilon)-f(x)}{d(x+t_n v_\varepsilon, x)}=\lim\limits_{n}\frac{\phi(x+t_n v_\varepsilon)-\phi(x)}{d(x+t_n v_\varepsilon, x)}$$
since $x+t_n v_\varepsilon\in N$ holds for every $n\in\mathbb N$. Now, Proposition \ref{prop:sufi[0,1]} applies to obtain that $C_\phi$ is an isometry. \end{proof}

Let us now exhibit a class of examples where Proposition \ref{prop:sufi[0,1]} applies. To this end, let us consider the following lemma which is well known by the speciallist. However, let us include a proof for the sake of completeness.

\begin{lemma}\label{lema:generaMcshane}
Let $M$ be a metric space and $N$ be a subset of $M$ containing the origin. Let $f:N\longrightarrow \mathbb R$ be a Lipschitz function and $g:M\longrightarrow \mathbb R$ Lipschitz too so that $\Vert g\Vert\leq \Vert f\Vert$ and such that $g\leq f$ on $N$.

Then there exists a Lipschitz extension $F:M \longrightarrow \mathbb R$ of $f$ such that $\Vert F\Vert=\Vert f\Vert$ and such that $g\leq F$ on $M$.
\end{lemma}

\begin{proof}
Given $x_0\in M\setminus N$, define
$$F(x_0):=\inf\limits_{x\in N} f(x)+\Vert f\Vert d(x,x_0).$$
It is known that the previous formula defines a norm preserving extension of $f$ to $M\cup\{x_0\}$ \cite[Theorem 1.33]{wea}. Let us prove that $g(x_0)\leq F(x_0)$. 
\[
\begin{split}
F(x_0)=\inf\limits_{x\in N} f(x)+\Vert f\Vert d(x_0,x)& \geq \inf\limits_{x\in N} g(x)+\Vert f\Vert d(x_0,x)\\
& \geq \inf\limits_{x\in N} g(x)+\Vert g\Vert d(x_0,x)\\
& \geq g(x_0).
\end{split}
\]
This shows that $f$ can always be extented to a new point, which finishes the proof.
\end{proof}

Let us now give examples of mappings $\phi$ satisfying the assumptions of Proposition 2.2, whose idea is encoded in the proof of \cite[Theorem 2.3]{kms}.

\begin{example}\label{exam:copias[0,1]}
Let $M$ be a metric space containing an isometric copy of $[0,1]$ and let $i:[0,1]\longrightarrow i([0,1])\subseteq M$ be an isometry. Consider $i^{-1}:i(M)\longrightarrow [0,1]\hookrightarrow \mathbb R$. If we apply Lemma \ref{lema:generaMcshane} to $f=i^{-1}$ and $g=0$ we get an extension $\phi: M\longrightarrow \mathbb R$ such that $\Vert \phi\Vert=1$ and $\phi(x)\geq 0$. Actually, since $\phi$ is a norm-one Lipschitz function we assume that $\phi(x)\in [0,1]$ (up to compose, for instance, with the function $\psi:\mathbb R_0^+\longrightarrow [0,1]$ by $$\phi(x):=\left\{\begin{array}{cc}
   x  &  x\leq 1,\\
   1 &  x>1,
\end{array} \right.$$
which is a norm-one Lipschitz function), so we will actually consider $\phi:M\longrightarrow [0,1]$. From the clear fact that $\phi\circ i$ is the identity map on $[0,1]$, so $\phi$ clearly satisfies the assumptions of Proposition 2.2.
\end{example}

Given a geodesic metric space $M$, a pair of points $x,y\in M$ and an $\alpha:[0,d(x,y)]\longrightarrow M$ such that $\alpha(0)=x$ and $\alpha(d(x,y))=y$, we will say that a norm-one Lipschitz function $P:M\longrightarrow [0,d(x,y)]$ is \textit{an inverse projection of $\alpha$} if $\Vert P\Vert=1$ and $P\circ \alpha$ is the indentity mapping on $[0,d(x,y)]$. Notice that the techniques involving Example \ref{exam:copias[0,1]} imply that given any $\alpha$ in the above conditions then an inverse projection of $\alpha$ always exists.

With this definition in mind, we can now give a neccesary condition for a composition operator to be isometric in the context of geodesic metric spaces.

\begin{theorem}\label{theo:maincondinece}
Let $M$ be a geodesic metric space, $N$ be a metric space and $\phi:N\longrightarrow M$ be a norm-one Lipschitz map. Assume that $C_\phi:\Lip(M)\longrightarrow \Lip(N)$ is isometric. Then, for every $x\neq y\in M$, every isometry $\alpha:[0,d(x,y)]\longrightarrow M$ connecting $x$ and $y$ and every inverse projection of $\alpha$ $P:M\longrightarrow [0,d(x,y)]$, the following holds: for every $t\in [0,d(x,y)]$, we can find a pair of sequences $x_n,y_n\in N$ such that $x_n\neq y_n$ for every $n\in\mathbb N$ and such that $P(\phi(x_n))\rightarrow t, P(\phi(y_n))\rightarrow t$ and 
$$\frac{P(\phi(x_n))-P(\phi(y_n))}{d(x_n,y_n)}\rightarrow 1.$$
\end{theorem}

\begin{proof}
Given $x,y\in M, x\neq y$ we can assume, with no loss of generality, that $d(x,y)=1$. Consider an isometry $\alpha: [0,1]\longrightarrow M$ connecting $x$ and $y$, consider an inverse projection of $\alpha$, say $P:M\longrightarrow [0,1]$. Then $C_P: \Lip([0,1])\longrightarrow \Lip(M)$ is an isometric composition operator because of Proposition \ref{prop:sufi[0,1]}. Thus, by assumptions, $C_\phi\circ C_P=C_{P\circ \phi}: \Lip([0,1])\longrightarrow \Lip(N)$ is an isometry. By Proposition \ref{proponece[0,1]} we obtain the desired pair of sequences.
\end{proof}

Inspired by Proposition \ref{prop:sufi[0,1]}, we can give a sufficient condition involving pointwise Lipschitz constants in the following sense.

\begin{theorem}\label{theo:maincondisufi}
Let $M$ be a geodesic metric space, $N$ be a connected metric space and let $\phi:N\longrightarrow M$ be a norm-one Lipschitz function with dense range. Assume that, for every pair of distinct points $x,y\in M$ , every isometry $\alpha:[0,d(x,y)]\longrightarrow M$ connecting the points $x$ and $y$, every inverse projection to $\alpha$ $P:M\longrightarrow [0,d(x,y)]$ and every $z\in P(\phi(N))$ there exists $x\in (P\circ \phi)^{-1}(z)$ such that
$$\justLip (P\circ \phi)(x)=1.$$
Then $C_\phi$ is an isometry.
\end{theorem}

\begin{proof}
Pick a norm-one Lipschitz function $f:M\longrightarrow \mathbb R$. Pick $x\neq y$ in $M$ such that $\frac{f(x)-f(y)}{d(x,y)}>1-\varepsilon^2$. Now consider an isometry $\alpha:[0,d(x,y)]\longrightarrow M$ such that $\alpha(0)=y$ and $\alpha(d(x,y))=x$. Now
$$1-\varepsilon^2<\frac{f(\alpha(d(x,y)))-f(\alpha(0))}{d(x,y)}=\frac{1}{d(x,y)}\int_0^{d(x,y)} (f\circ \alpha)'(t)\ dt.$$
Now we can find $z\in P(\phi(N))$ such that $(f\circ \alpha)'(z)>1-\varepsilon$ (notice that the denseness of $\phi(N)$ implies that $P(\phi(N))$ contains $]0,d(x,y)[$ because $N$ is connected).  By assumption there exists $x\in N$ such that $P(\phi(x))=z$ and a sequence $x_n$ in $N$ such that $P(\phi(x_n))\rightarrow z$ and such that 
$$\frac{P(\phi(x_n))-P(\phi(x))}{d(x_n,x)}\rightarrow 1.$$
On the other hand, 
$$\frac{(f\circ\alpha)(P(\phi(x_n)))-(f\circ\alpha)(P(\phi(x)))}{P(\phi(x_n))-P(\phi(x))}\rightarrow (f\circ \alpha)'(z)>1-\varepsilon.$$
Consequently
\[
\begin{split}
\Vert C_\phi (f)\Vert&  \geq \limsup \frac{(f\circ\alpha)(P(\phi(x_n)))-(f\circ\alpha)(P(\phi(x)))}{d(x_n,x)}\\
& =\limsup \frac{(f\circ\alpha)(P(\phi(x_n)))-(f\circ\alpha)(P(\phi(x)))}{P(\phi(x_n))-P(\phi(x))} \frac{P(\phi(x_n))-P(\phi(x))}{d(x_n,x)}\\ &  \geq 1-\varepsilon.
\end{split}
\]
Since $\varepsilon>0$ was arbitrary we conclude the desired result.
\end{proof}

Let us end with the following remark.

\begin{remark}
In the proof of Theorem \ref{theo:maincondisufi} the use of inverse projections has been necessary in order to work, in some sense, with directional derivatives of $\phi$ throughout isometric curves $\alpha:[0,1]\longrightarrow M$ and, in order to construct such projections $P$, the assumptions of $M$ being geodesic has been essential. In general, there are conditions on a complete metric space to ensure that $M$ is geodesic, and in the end where Theorems \ref{theo:maincondinece} and \ref{theo:maincondisufi} apply, if we assume $M$ being length (see e.g. \cite[Proposition 2.9]{ikw} or \cite[Corollary 4.4]{gpr}), being particularly interesting the case when $M$ is compact. However, there are length metric spaces with are not geodesic \cite[Example 2.4]{ikw}.  We do not know any sufficient condition for a composition operator $C_\phi$ to be isometric if $M$ is length (and it is not geodesic).
\end{remark}

\textbf{Acknowledgements:} The author is deeply grateful to Antonio Jim\'enez-Vargas for sending him the final version of \cite{jiva}. He also thanks Rafael Chiclana, Luis C. Garc\'ia-Lirola and Colin Petitjean for their comments that improved the final version of the paper.

\end{document}